\documentclass[11pt,a4paper]{article}
\usepackage{amscd}
\usepackage{enumerate}

\usepackage{amsmath, amssymb, latexsym}
\usepackage{amsfonts}
\usepackage{amsthm}
\usepackage{relsize}
\usepackage{setspace}
\usepackage{geometry}
\usepackage{url}
\usepackage{xspace}
\usepackage{tocloft}
\usepackage{graphics}
\usepackage{graphicx}
\usepackage{lscape}
\usepackage{microtype}
\usepackage{ulem}

\usepackage[usenames, dvipsnames]{color}
\usepackage[utf8]{inputenc}
\usepackage{tikz}

\usepackage[pagebackref=true]{hyperref}
\usepackage[alphabetic]{amsrefs}
\usepackage[english]{babel}

\usepackage{authblk}

\newtheorem{theorem}{Theorem}[section]
\newtheorem{proposition}[theorem]{Proposition}

\newtheorem{lemma}[theorem]{Lemma}

\theoremstyle{definition}
\newtheorem{definition}[theorem]{Definition}

\newtheorem{remark}[theorem]{Remark}

\theoremstyle{problem}
\newtheorem{problem}[theorem]{Problem}

\newcommand{\Aut}{\mathrm{Aut}}

\newcommand{\RR}{\mathbf{R}}
\newcommand{\ZZ}{\mathbf{Z}}

\newcommand{\NN}{\mathbf{N}}
\newcommand{\Cen}{\mathrm{Cent}}
\newcommand{\CAT}{\mathrm{CAT}}
\newcommand{\cat}{$\mathrm{CAT}(0)$\xspace}
\newcommand{\Min}{\mathrm{Min}}
\newcommand{\Ch}{\mathrm{Ch}}
\newcommand{\Stab}{\mathrm{Stab}}
\newcommand{\Fix}{\mathrm{Fix}}

\newcommand{\dist}{\operatorname{dist}}

\newcommand{\pr}{\mathrm{pr}}
\newcommand{\Eucl}{\mathrm{\bold{Eucl}}}
\newcommand{\Sph}{\mathrm{\bold{Sph}}}
\newcommand{\nEsph}{\mathrm{\bold{nEsph}}}
\newcommand{\nSph}{\mathrm{\bold{nSph}}}
\newcommand{\A}{\mathrm{\bold{A}}}
\newcommand{\ch}{\mathrm{\bold{c}}}
\newcommand{\eqstop}{\ensuremath{\, \text{.}}}

\def\og{\leavevmode\raise.3ex\hbox{$\scriptscriptstyle\langle\!\langle$~}}
\def\fg{\leavevmode\raise.3ex\hbox{~$\!\scriptscriptstyle\,\rangle\!\rangle$}}

\def\crr{\texttt{r}}

\def\cll{\texttt{l}}
\def\cLL{\texttt{L}}

\title{The flat closing problem for buildings}

\author{Corina Ciobotaru\thanks{Supported  by the FRIA; corina.ciobotaru@uclouvain.be}}

\date{First draft: November 5, 2013; Accepted: March 20, 2014}

\begin{document}

\maketitle

\begin{abstract}
Using the notion of a strongly regular hyperbolic automorphism of a locally finite Euclidean building, we prove that any (not necessarily discrete) closed, co-compact subgroup of the type-preserving automorphisms group of a locally finite  general non-spherical building contains a compact-by-$\ZZ^{d}$ subgroup, where $d$ is the dimension of a maximal flat.
\end{abstract}

\renewcommand{\thefootnote}{}

\newcounter{qcounter}

\section{Introduction}

In 1972, Prasad and Raghunathan~\cite{PR72} proved the following result (see \cite[Corollary~2.9, Lemma~1.15]{PR72}): Let $G$ be a semi-simple real Lie group of rank $r$ (which may admit compact factors). Let $\Gamma < G$ be a lattice. Then $\Gamma$ contains an abelian subgroup of rank $r$.

To obtain this result, they use in the first place the existence in $\Gamma$ of a so-called $\RR$--hyper-regular element (see Prasad--Raghunathan~\cite[Definition~1.1 and Theorem~2.5]{PR72}). Being also an $\RR$--regular element (see \cite[Remark~1.2]{PR72}, Steinberg~\cite{St65}), an $\RR$--hyper-regular $g \in G$ inherits the property that its $G$--centralizer $\Cen_{G}(g)$ contains a unique maximal $\RR$--split torus of $G$ of dimension the $\RR$--rank of $G$ (see \cite{St65} or \cite{PR72}). Using this, the last step of the Prasad--Raghunathan strategy is to show that $\Cen_{G}(g) / \Cen_{\Gamma}(g)$ is compact if $g$ is an  $\RR$--hyper-regular element (see \cite[proof of Lemma 1.15]{PR72}). This is obtained in the following way. If $\Gamma$ is a uniform lattice then one can use a Selberg's lemma (see \cite[Lemma~1.10]{PR72}). If not, then \cite[Theorem~1.14]{PR72} gives the desired result.

The above result of Prasad and Raghunathan can be related,  in the setting of $\CAT(0)$ spaces, to the following question of Gromov~\cite{Gro93} (also known as the \textbf{flat closing problem}) (see \cite[Section 6.$B_{3}$]{Gro93}). From another point of view, the flat closing problem is a converse to the Flat Torus Theorem.

First, let us recall the following basic notions from the setting of $\CAT(0)$ spaces.

\begin{definition}
\label{def:: flat-CAT}
Let $X$ be a proper $\CAT(0)$ space and $G$ be a not necessarily discrete locally compact group acting continuously and properly by isometries on $X$. By a \textbf{geometric flat} $F \subset X$ of dimension $d$ (or a \textbf{$\mathbf{d}$--flat}) we mean a closed convex subset of $X$ which is isometric to the Euclidean $d$--space and where $d\geq 2$. Moreover, we say that a subset $Y \subset X$ is \textbf{periodic} if $\Stab_{G}(Y)$ acts co-compactly on $Y$.

When $\gamma$  is a hyperbolic isometry of $X$ we denote
$$\Min(\gamma):=\{x \in X \; | \; \dist_{X}(x, \gamma(x))=|\gamma| \},$$ 
where $|\gamma|:= \inf\limits_{x \in X} \dist_{X}(x, \gamma(x))$ denotes the translation length of $\gamma$. 
\end{definition}

\begin{problem}(The flat closing problem)
\label{prob:flat_clos-prob}
Let $X$ be a proper $\CAT(0)$ space and $\Gamma$ be a discrete group acting continuously, properly and co-compactly by isometries on $X$. Does the existence of a $d$--flat $F \subset X$ imply that $\Gamma$ contains a copy of $\ZZ^{d}$?
\end{problem}

\begin{remark}
\label{rem:compact_by_abelian} In the hypotheses of Problem~\ref{prob:flat_clos-prob}, note that if $F$ is a periodic $d$--flat then indeed, by Bieberbach Theorem, $\Stab_{\Gamma}(F)$ contains a copy of $\ZZ^{d}$ and therefore $\Gamma$ contains a subgroup which is virtually $\ZZ^{d}$. Moreover, if we replace $\Gamma$ with a not necessarily discrete locally compact group $G$ acting continuously,  properly and co-compactly by isometries on $X$ and $F$ is a periodic $d$--flat with respect to the $G$--action then $G$ contains a compact--by--$\ZZ^{d}$ subgroup.
\end{remark}

To attack Gromov's Problem~\ref{prob:flat_clos-prob}, the natural strategy would be to construct periodic flats or more generally periodic subspaces of the form $Y=F \times C \subset X$, where $F$ is a $d$--flat and $C$ is a compact set. This could be done by showing that $\Gamma$ contains a hyperbolic element $\gamma$ which is `regular' (see Definition~\ref{def::reg-hyp}) and then to consider $Y:=\Min(\gamma)=F \times C$, where $F$ is a flat and $C$ is a compact set. To conclude from here the flat closing problem it would be enough to show that the centralizer of $\gamma$ in $\Gamma$, namely, $\Cen_{\Gamma}(\gamma) <\Stab_{\Gamma}(Y)$,  acts co-compactly on $Y$ and that $\Stab_{\Cen_{\Gamma}(\gamma)}(F) \leq\Stab_{\Gamma}(F)$ has co-compact action on $F$. Thus, $\Stab_{\Gamma}(F)$ would contain a copy of $\ZZ^{d}$ and the Remark~\ref{rem:compact_by_abelian} is applied.

Notice that the above strategy is analogous to the one used in the Prasad--Raghunathan result, where the $\RR$--split torus is replaced with a flat.  Moreover, this strategy, including the existence of `regular' elements, is successfully implemented for example in Caprace--Zadnik~\cite{CaZa}, where it is proved that any discrete group acting properly and co-compactly on a decomposable locally compact $\CAT(0)$ space, which admits in addition the geodesic extension property (i.e. every geodesic segment is contained in a bi-infinite geodesic line) contains virtually $\ZZ^{d}$, where $d$ is the number of indecomposable de Rham factors. Remark that in general, the dimension of a maximal flat of a \cat space is bigger than the number of the indecomposable de Rham factors.

This article proposes to answer further Gromov's Problem \ref{prob:flat_clos-prob} in the case when $G$ is a not necessarily discrete locally compact group acting continuously, properly and co-compactly by type-preserving automorphisms on a locally finite general non-spherical building. We stress here that we cannot apply directly the result of Caprace--Zadnik~\cite{CaZa}. Firstly, because in the usual Davis realization, buildings are $\CAT(0)$ spaces where in general the geodesic extension property is not fulfilled. Secondly, the group $G$ considered in this article is not necessarily discrete. Thirdly, we exploit two other ingredients that are missing in the context of Caprace--Zadnik~\cite{CaZa}: the case of Euclidean buildings and the reduction to products in the non-affine part. Still, the strategy is the same and uses the existence of `strongly regular hyperbolic automorphisms' (see Definition \ref{def::str-reg-hyp} below) acting on locally finite Euclidean buildings. We obtain:

 \begin{theorem}[See Theorem~\ref{thm:flat-clos-conj}]
\label{thm:flat-clos-conj_first}
Let $\Delta$ be a locally finite building of not-finite type $(W,S)$, but with $S$ being finite. Let $G$ be a closed, not necessarily discrete, type-preserving subgroup of $\Aut(\Delta)$, with co-compact action on $\Delta$. Then $G$ contains a compact--by--$\ZZ^{d}$ subgroup, where $d$ is the dimension of a maximal flat of $\Delta$.
\end{theorem}
 
\section{Preliminaries}

We start by recalling briefly some basic definitions and results from Caprace--Ciobotaru~\cite{CaCi} on strongly regular hyperbolic automorphisms acting on locally finite Euclidean buildings, as they are used further in this paper. For some notation see also Definition \ref{def:: flat-CAT} from the Introduction.

\begin{definition}
\label{def::reg-hyp}
Let $X$ be a $\CAT(0)$ space and $\gamma$ be a hyperbolic isometry of $X$. We say that $\gamma$ is a \textbf{regular hyperbolic isometry} if $\Min(\gamma)$ is a bounded Hausdorff distance from a maximal flat of $X$. 
\end{definition}

Specified to the case of locally finite Euclidean buildings we have the following stronger definition introduced in Caprace--Ciobotaru~\cite{CaCi}. 

\begin{definition}
\label{def::str-reg-hyp}
Let $\Delta$ be a Euclidean building and $\gamma \in \Aut(\Delta)$ be a type-preserving automorphism. We say that  $\gamma$ is a \textbf{strongly regular hyperbolic automorphism} if $\gamma$ is a hyperbolic isometry and the two endpoints of one (and hence all) of its translation axes lie in the interior of two opposite chambers of the spherical building at infinity. In particular $\Min(\gamma)$ is an apartment of $\Delta$  and is uniquely determined. 

In addition, if $\ell$ is geodesic line of the Euclidean building $\Delta$, we say that $\ell$ is \textbf{strongly regular} if its endpoints lie in the interior of two opposite chambers of the spherical building at infinity of $\Delta$.
\end{definition}

\begin{definition}
Let $X$ be a \cat space. A geodesic line $\ell$ in X is said to have \textbf{rank-one} if it does not bound a flat half-plane. An isometry $\gamma \in \mathrm{Is}(X)$ is said to have \textbf{rank-one} if it is hyperbolic and if some (and hence any) of its axes has rank-one.
\end{definition}

Used in the sequel, we have the following two results from Caprace--Ciobotaru~\cite{CaCi}.

\begin{lemma}[See Lemma~2.6 from \cite{CaCi}]
\label{lem::existance_reg_element}
Let $(W,S)$ be a  Euclidean Coxeter system. Then $W$ contains strongly regular hyperbolic elements.
\end{lemma}

\begin{proposition}[See Proposition~2.9 from \cite{CaCi}]
\label{prop:TechSRH}
Let $X$ be a proper \cat space, $G \leq \mathrm{Is}(X)$ be any subgroup of isometries and $\rho \colon \RR \to X$ be a geodesic map. Assume there is an increasing sequence $\{t_n\}_{n \geq 0}$ of positive real numbers tending to infinity such that $\sup_n d(\rho(t_n), \rho(t_{n+1})) < \infty$ and that the set $\{\rho(t_n)\}_{n \geq 0}$ falls into finitely many $G$--orbits, each of which is moreover discrete.

Then there is an increasing sequence $\{f(n)\}_{n}$ of positive integers such that, for all $n> m > 0$, there is a hyperbolic isometry $h_{m, n} \in G$ which has a translation axis containing the geodesic segment $[\rho(t_{f(m)}), \rho(t_{f(n)})]$.

In addition, if the geodesic line $\rho(\RR)$ is the translation axis of a rank-one element of $\mathrm{Is}(X)$, then for each fixed $m>0$ there exists $N_{m}>0$ such that for every $n>N_{m}$ the isometry $h_{m,n}$ is a rank-one element. Moreover, if $X$ is a locally finite Euclidean building and the geodesic line $\rho(\RR)$ is strongly regular, then $h_{m, n}$ is a strongly regular hyperbolic automorphism. 
 
\end{proposition}

\begin{proof}
We give a proof only for the very last assertion of the proposition, this not being part of \cite[Proposition~2.9]{CaCi}.

Let $m$ be fixed. Suppose the contrary, namely, for every $k>0$ there exists $n_{k} \geq k$ such that $h_{m,n_k}$ is not a rank-one element. This means that the translation axis of $h_{m,n_k}$ containing the geodesic segment $[\rho(t_{f(m)}), \rho(t_{f(n_k)})]$ is contained in the boundary of a flat half-plane. Therefore, as $k \to \infty$ we obtain that the geodesic ray $[\rho(t_{f(m)}), \rho(\infty))$ is contained in the boundary of a flat half-plane as well. As the space $X$ is proper and the geodesic line $\rho(\RR)$ is the translation axis of a rank-one hyperbolic element, we obtain a contradiction with the fact that the diameter of the projection on $\rho(\RR)$ of every closed metric ball in $X$, which is moreover disjoint from $\rho(\RR)$, must be bounded above by a fixed constant. The conclusion follows.
\end{proof}

\section{The proof of the main theorem}
\label{sec::Gromov-problem}

Before proceeding to the proof of Theorem~\ref{thm:flat-clos-conj_first}, we recall some general facts about buildings and we fix some notation. Let $\Delta$ be a locally finite general building of type $(W,S)$, with $S$ finite. Fix from now on a chamber $\ch$ in $\Ch(\Delta)$. Let $W=W_{1} \times W_{2} \times \cdots \times W_{k}$ be the direct product decomposition of $W$ in irreducible Coxeter systems $(W_{i}, S_{i})$. Thus $S=S_{1} \sqcup S_{2} \sqcup \cdots \sqcup S_{k}$ is a disjoint union. Denote by $\Delta_{i}$ the $W_{i}$--residue in $\Delta$ containing the fixed chamber $\ch$. From Ronan~\cite[Theorem 3.10]{Ron89} we have that $\Delta \cong \Delta_{1} \times \cdots \times \Delta_{k}$. For what follows, we use the following notation:
 \[
\Eucl=\{i \in \{1,\cdots, k\} \;| \; (W_{i}, S_{i}) \text{ is Euclidean}\}, \text{ }\Sph=\{i \in \{1,\cdots, k\} \;| \; (W_{i}, S_{i}) \text{ is finite}\},
\]
\[
\nSph= \{1,\cdots, k\} \setminus \Sph \text{\ \ \ \ and \ \ \ } \nEsph=\{1,\cdots, k\} \setminus (\Eucl  \cup \Sph) \eqstop
\]
Accordingly, we use the notation $\Delta_{\A}:=\prod\limits_{i \in \A}\Delta_{i}$ and $W_{\A}:=\prod\limits_{i \in \A}W_{i}$, where $\A$ is one of the sets $\Eucl$, $\Sph$, $\nSph$ or $\nEsph$. Moreover, for every $ i \in \{1,\cdots, k\}$ we denote by $\pr_{i} \colon \Delta \to \Delta_{i}$ the projection map on $\Delta_{i}$ and by abuse of notation we write $\pr_{i}(\gamma)$ to represent a type-preserving element $\gamma \in \Aut(\Delta)$ acting on the factor $\Delta_{i}$.

The first step towards the main theorem is given by the next proposition which uses an argument of Hruska and Kleiner~\cite[Lemma 3.1.2]{HK05}.

\begin{proposition}
\label{prop:compact_action}
Let $\Delta$ be a locally finite building of not-finite type $(W,S)$, but with $S$ being finite. Let $G$ be a not necessarily discrete, type-preserving subgroup of $ \Aut(\Delta)$ acting co-compactly on $\Delta$ and let $\mathcal{R}$ be any residue in $\Delta$ containing the chamber $\ch$. Then $\Stab_{G}(\mathcal{R})$ acts co-compactly on $\mathcal{R}$. 
\end{proposition}

\begin{proof}
Let us denote by $(W', S')$ the type of the residue $\mathcal{R}$, where $W' \leq W$ and $S' \subset S$. 
Take $K \subset \Delta$ be a compact fundamental domain corresponding to the action of $G$ and containing the chamber $\ch$. Let $(g_{i})_{i \in I}$ be a subset of $G$ such that $\mathcal{R} \subset \bigcup\limits_{i \in I} g_{i}(K)$. Because  $\mathcal{R}$ is a residue containing $\ch$ and $G$ is type-preserving, for every $i \in I$, $g_{i}^{-1}(\mathcal{R})$ is a residue of the same type as $\mathcal{R}$, containing the chamber $g_{i}^{-1}(\ch)$ and intersecting $K$.

Notice that a compact set $K$ of a (not necessarily locally finite) building always has a finite number of chambers. Therefore, $K$ intersects a finite number of $(W', S')$--type residues of $\Delta$. We conclude that there is a finite number of left cosets of the form $\Stab_{G}(\mathcal{R}) g_{i}$, with $i \in I$. Denote by $\{ g_{1}, \cdots, g_{t}\} \subset \{ g_{i} \}_{i \in I}$  the finite set of representatives of these left cosets. We obtain that $\mathcal{R}$ is covered by $\bigcup\limits_{j \in \{1,\cdots, t\}} \Stab_{G}(\mathcal{R}) g_{j}(K)$.

Because $g_{j}(K)$ is compact, let $K'$ be a compact in $\Delta$ such that $\bigcup\limits_{j=1}^{t}g_{j}(K) \subset K'$. Thus $\Stab_{G}(\mathcal{R})K'$ covers $\mathcal{R}$ and the conclusion follows.

\end{proof}

The second step is to find a hyperbolic element in $\Stab_{G}(\Delta_{\nSph})$. This is given by the following proposition. 

\begin{proposition}
\label{prop:ext-hyp-iso}
Let $\Delta$ be a locally finite building of not-finite type $(W,S)$, but with $S$ being finite. Let $G$ be a not necessarily discrete, type-preserving subgroup of $\Aut(\Delta)$ acting co-compactly on $\Delta$. Then there exists a hyperbolic element $\gamma \in \Stab_{G}(\Delta_{\nSph})$ such that $\pr_{i}(\gamma)$ is a strongly regular hyperbolic element if $i \in \Eucl$ and a rank-one isometry if $i \in \nEsph$.
\end{proposition}

\begin{proof}
First, remark that $G$ acts co-compactly on $\Delta_{\nSph}$. Take $i \in \Eucl$. By Lemma~\ref{lem::existance_reg_element}, let $\cll_{i}$ be a strongly regular geodesic line contained in some apartment of $\Delta_{i}$, constructed using a strongly regular hyperbolic element $\gamma_{i}$ of $W_i$. The line $\cll_{i}$ is thus contained in a unique apartment of $\Delta_{i}$. Denote by $\{v_{i,j}\}_{j \in \ZZ}$ a bi-infinite sequence of points in $\cll_{i}$ such that $\gamma_{i}(v_{i,j})=v_{i,j+1}$, for every $j \in \ZZ$. For example, the points $\{v_{i,j}\}_{j \in \ZZ} \subset \cll_{i}$ can be taken to be special vertices of $\Delta_{i}$, of the same type.

For $i \in \nEsph$ we have a similar construction. Following Caprace--Fujiwara~\cite[Proposition~4.5]{CaFu}, denote by $\crr_{i}$ a rank-one geodesic line given by a rank-one hyperbolic element $h_{i} \in W_i$. Denote by $\{t_{i,j}\}_{j \in \ZZ}$ a bi-infinite sequence of points of $\crr_{i}$ such that $h_{i}(t_{i,j})=t_{i,j+1}$, for every $j \in \ZZ$.

Because $\prod\limits_{i \in \Eucl}\cll_{i} \times \prod\limits_{i \in \nEsph}\crr_{i}$ is a flat of $\Delta_{\nSph}$, we consider in its interior the infinite geodesic line determined by the following sequence of points 
$$\{o_{j}:=\prod\limits_{i \in \Eucl}v_{i,j} \times \prod\limits_{i \in \nEsph}t_{i,j}\}_{j \in \ZZ} \subset \Delta_{\nSph} \eqstop$$ Denote the resulting geodesic line by $\cLL$ and observe that, by defining $h:=\prod\limits_{i \in \Eucl}\gamma_{i} \times \prod\limits_{i \in \nEsph}\gamma_{i} $, we have that $h \in W_{\nSph}$ and $h(o_{j})=o_{j+1}$, for every $j \in \ZZ$. 

We are now ready to proceed in finding the desired hyperbolic element in $G$. Apply Proposition~\ref{prop:TechSRH} to our geodesic line $\cLL$, to the sequence of points $(o_{j})_{j \in \NN}$ and to the $\Stab_{G}(\Delta_{\nSph})$--action on $\Delta_{\nSph}$ (which is co-compact). As we are working with a locally finite building, all hypotheses of Proposition~\ref{prop:TechSRH} are fulfilled. We obtain thus a sequence $\{f(n)\}_{n\geq 0}$ and a sequence of hyperbolic elements $\{\gamma_{m,n}\}_{0<m<n} \subset \Stab_{G}(\Delta_{\nSph})$ such that every $\gamma_{m,n}$ has a translation axis containing the geodesic segment $[o_{f(m)}, o_{f(n)}]$. By this construction we obtain that $\pr_{i}(\gamma_{m,n}(o_{f(m)}))=\pr_{i}(o_{f(n)})$, for every $i \in \nSph$ and every element $\gamma_{m,n}$. Applying again Proposition~\ref{prop:TechSRH}, there exists a hyperbolic element $\gamma_{m,n}$ such that $\pr_{i}(\gamma_{m,n})$ is a strongly regular hyperbolic element if $i \in \Eucl$ and a rank-one isometry if $i \in \nEsph$. The conclusion follows. 
\end{proof}

Before starting the proof of the main Theorem~\ref{thm:flat-clos-conj}, let us mention that the power $d$ of the `compact--by--$\ZZ^{d}$ subgroup' appearing in the conclusion of the theorem is maximal with respect to the $d$--flats of the building $\Delta$. The maximality is explained by the following result.

\begin{proposition}(See Caprace~\cite[Proposition~3.1]{Ca09})
\label{prop::flat_residue}
Let $F$ be a maximal $d$--flat of a locally finite general building $\Delta$.  Then there exists a residue $\mathcal{R} \subset \Delta$ of type $(W_{\mathcal{R}}, S_{\mathcal{R}})$ such that  $d = \sum\limits_{i \in \Eucl_{\mathcal{R}}} n_{i} +\vert \nEsph_{\mathcal{R}} \vert$, where $\Eucl_{\mathcal{R}}$ and $\nEsph_{\mathcal{R}}$ correspond to the residue $\mathcal{R}$. 
\end{proposition}

Therefore, by Propositions~\ref{prop:compact_action} and~\ref{prop::flat_residue}, to answer Gromov's flat closing problem in the case of a locally finite general non-spherical building it is enough to prove the following theorem. 

\begin{theorem}
\label{thm:flat-clos-conj}
Let $\Delta$ be a locally finite building of not-finite type $(W,S)$, with $S$ being finite, and $G$ a closed, not necessarily discrete, type-preserving subgroup of $\Aut(\Delta)$ with co-compact action. Then $G$ contains a compact--by--$\ZZ^{d}$ subgroup, where $d:=\sum\limits_{i \in \Eucl} n_{i} +\vert \nEsph \vert$. 
\end{theorem}

\begin{proof}
Let $\gamma \in \Stab_{G}(\Delta_{\nSph})$ be a hyperbolic element given by Proposition~\ref{prop:ext-hyp-iso}. We have that $\Min(\gamma)=\prod_{i \in \Eucl} \RR^{n_{i}} \times \prod_{j \in \nEsph} (\RR\times C_{j})$, where $n_{i}$ is the Euclidean dimension of the corresponding building $\Delta_{i}$ and $C_{j}$ is a compact, convex subset of the corresponding building $\Delta_{j}$. Let $F:=\prod_{i \in \Eucl} \RR^{n_{i}} \times \prod_{i \in \nEsph} \RR \subset \Min(\gamma)$. By Ruane~\cite[Theorem 3.2]{Rua01}, whose proof works also for not necessarily discrete groups, we have that $\Cen_{G}(\gamma)$ stabilizes and acts co-compactly on $\Min({\gamma})$. As $\prod_{i \in \nEsph}C_{i}$ is compact, using the argument of Hruska--Kleiner~\cite[Lemma 3.1.2]{HK05}, recalled in the proof of Proposition~\ref{prop:compact_action}, and the fact that we are working with $\CAT(0)$ cellular complexes we obtain that $\Stab_{\Cen_{G}(\gamma)}(F)$ acts co-compactly on $F$. In particular, we obtain that the action of $\Stab_{G}(F)$ on $F$ is properly discontinuous and co-compact. Therefore, the group $\Stab_{G}(F)/\Fix_{G}(F)$ is virtually isomorphic with $\ZZ^{d}$, where $d:=\sum\limits_{i \in \Eucl} n_{i} +\vert \nEsph \vert$. Since $G$ is closed, the group $\Fix_{G}(F)$ is compact. Therefore, the group $G$ contains a compact--by--$\ZZ^{d}$ subgroup. This concludes the proof of the theorem.
\end{proof}

\subsection*{Acknowledgements} We would like to thank Pierre-Emmanuel Caprace for proposing this question, for his comments and useful discussions, Ga\v{s}per Zadnik for some useful explanations about $\CAT(0)$ spaces and the referee for his/her useful comments. The author is supported by the FRIA.

\end{document}